\documentclass{amsart}
\usepackage{amssymb, amsmath}
\usepackage{fullpage}
\usepackage{array}
\usepackage{mathrsfs}
\usepackage{xcolor}
\usepackage{hyperref}
\usepackage{graphicx,color}
\usepackage{float}
\usepackage{tikz}
\usepackage{comment}
\usepackage{amsaddr} 
\usetikzlibrary{decorations.pathreplacing,calligraphy}

\newcommand{\R}{\mathbb{R}}
\newcommand{\N}{\mathbb{N}}

\newcommand{\Z}{\mathbb{Z}}

\DeclareMathOperator{\ase}{A_{\sigma \eta}}
\DeclareMathOperator{\as1}{A_{\sigma1}}
\DeclareMathOperator{\SL(2)}{SL(2,\R)}

\newtheorem{teorema}{Theorem}[section] 

\newtheorem{lema}[teorema]{Lemma}

\newcolumntype{C}[1]{>{\centering\let\newline\\\arraysetminus\hspace{0pt}}m{#1}}

\usepackage[utf8]{inputenc}
\usepackage{amsfonts}
\usepackage{amssymb}
\usepackage{verbatim}
\usepackage[pagewise]{lineno}
\numberwithin{equation}{section}

\title{Discontinuity example for the Lyapunov exponents on the boundary of the uniformly hyperbolic set}

\author{Raquel Saraiva$^{1}$}
\address{\small Instituto de Ciências Exatas, Universidade Federal de Minas Gerais, Av. Antônio Carlos 6627, \\
\small Belo Horizonte, 31270-901, MG, Brazil.\\
\small (e-mail: rss07@ufmg.br)}

\thanks{\footnotesize $^{1}$ Supported by FAPEMIG}

\begin{document}

\begin{abstract}
We present an example of a discontinuity point for the Lyapunov exponents when viewed as a function of the cocycle in a topology finer than the $C^0$-topology. The linear cocycle taking values in $\SL(2)$ is locally constant, defined over a Bernoulli shift and it lies on the boundary of the uniformly hyperbolic set. In particular, we show that it can be approximated, in the $C_{\delta-\log}$-topology, by cocycles whose Lyapunov exponents vanish.
\end{abstract}

\maketitle   

\section{Introduction}
Lyapunov exponents were introduced by Lyapunov in the late 19th century as a tool to investigate the stability of solutions to ordinary differential equations \cite{lya92}. Subsequent developments by the results of Furstenberg-Kesten \cite{furstenberg} and Oseledets \cite{oseledets}, established them  in the field of ergodic theory. They also play a crucial role in the study of smooth dynamical systems, especially in the context of $C^k$-diffeomorphisms on manifolds. 
In this sense, Pesin developed the theory of non-uniform hyperbolicity \cite{pes76}, \cite{pes77}. A $C^{1+\alpha}$-diffeomorphism preserving an ergodic measure with non-zero Lyapunov exponents admits, almost everywhere, well-defined local stable and unstable manifolds. This result extends several fundamental properties of uniformly hyperbolic systems (Axiom A). 

Beyond the hyperbolic theory, Lyapunov exponents exhibit deep connections with other fundamental invariants, notably metric entropy and the fractal dimension of invariant measures, as demonstrated in the works of Barreira, \cite{bar12}, Ledrappier, \cite{led86}, and Young, \cite{you82}. Motivated by the inherent analytical difficulties of smooth dynamics, the concept of Lyapunov exponents has been successfully extended to non-smooth frameworks, where they are studied via linear cocycles over continuous base dynamical systems. It is precisely within this setting that the present work is situated. 

More specifically, we are interested in their regularity properties, which is an important question in the theory of Lyapunov exponents. Oseledets’ theorem ensures that these exponents depend measurably on both the cocycle and the invariant measure. In contrast, we study their continuity with respect to variations only in the cocycle, with particular emphasis on $\mathrm{SL}(2,\mathbb{R})$-valued cocycles.

For random matrix products in dimension two, continuity has been established in both classical stochastic settings: the Bernoulli case by Bocker and Viana~\cite{bockerviana} and the Markov case by Malheiro and Viana~\cite{malheiroviana}. In the higher dimensional context the picture is more subtle, but a recent announcement by Avila, Eskin, and Viana~\cite{aev} concerning i.i.d. random products of $m\times m$ matrices indicates that continuity extends beyond the two-dimensional setting.

Although continuity of Lyapunov exponents has been verified in those specific settings, it is far from being a general feature. Indeed, discontinuities arise frequently. Fixing an ergodic measure $\mu$, Bochi-Mañé's Theorem \cite{bochi} provides the information about continuity of Lyapunov exponents of an aperiodic system on a compact metric space $(M, f,\mu)$. This theorem establishes that, within the space of continuous cocycles $C^0(M,\mathrm{SL}(2,\mathbb{R}))$ over $(M,f,\mu)$, the points of continuity for Lyapunov exponents are exactly those corresponding either to uniformly hyperbolic cocycles or to cocycles with vanishing Lyapunov exponents. In particular, any typical cocycle with nonzero Lyapunov exponents can be approximated in the $C^0$-topology by cocycles with zero Lyapunov exponents. 

It is a classical fact that the set of uniformly hyperbolic cocycles forms an open subset of $C^0(M,\mathrm{SL}(2,\mathbb{R}))$, \cite{viana}. A cocycle $A \in C^0(M,\mathrm{SL}(2,\mathbb{R}))$ belongs to the boundary of the uniformly hyperbolic set, denoted by $\partial \mathcal{UH}$, if and only if, there exist two sequences $(A_n)_{n \in \mathbb{N}} \subset \mathcal{UH}$ and $(\tilde{A}_n)_{n \in \mathbb{N}} \subset C^0(M,\mathrm{SL}(2,\mathbb{R})) \setminus \mathcal{UH}$ such that $A_n \to A$ and $\tilde{A}_n \to A$ in the $C^0$-topology. Moreover, the Bochi-Mañé's theorem provides a sharper description of points in the boundary of this set, whenever $A \in \partial \mathcal{UH}$ there exists a sequence $(B_n)_{n \in \mathbb{N}} \subset C^0(M,\mathrm{SL}(2,\mathbb{R})) \setminus \mathcal{UH}$ converging to $A$ in the $C^0$-topology such that each $B_n$ has vanishing Lyapunov exponents.
As a consequence, this dichotomy shows that a cocycle lying in $\partial \mathcal{UH}$ is a continuity point for the Lyapunov exponents if and only if its Lyapunov exponents vanish. This provides a full description of the continuity behavior at the boundary of the uniformly hyperbolic set.

Therefore, these considerations naturally motivate the study of coarser topologies. Considering the $C^{\alpha}$-topology, a well studied topology, Backes, Brown and Butler \cite{bbb} established that, when restricted to the set of cocycles satisfying the quasi-conformality condition, known as \textit{fiber-bunching}, over hyperbolic homeomorphism systems on compact metric space and with invariant ergodic probability measure with local product structure, the Lyapunov exponents vary continuously. Their result applies to cocycles admitting invariant holonomies, a property ensured by the fiber-bunching condition. 

Within this setting, the continuity properties become more flexible. In particular, cocycles lying in the boundary of the uniformly hyperbolic set may still be continuity points for Lyapunov exponents. Indeed, any cocycle in $\partial \mathcal{UH}$ satisfying the fiber-bunching condition is necessarily a point of continuity, in contrast with the $C^0$-topology scenario.

In this paper, we investigate new examples of cocycles that lie on the boundary of uniformly hyperbolic set. These examples illustrate the subtle transition between uniform and non-uniform hyperbolicity, providing insight into how Lyapunov exponents behave near the boundary of hyperbolic settings. In order to advance this question and motivated by the goal of understanding the behavior of cocycles in broader settings beyond the Hölder framework, we introduce intermediate topologies, finer than the 
$C^0$-topology, which provide a suitable setting for analyzing the continuity problem of Lyapunov exponents in this particular example. 

More precisely, we consider a particular cocycle $A \in \partial \mathcal{UH}$, whose definition is presented in Section \ref{sec3} and we consider the $C_{\delta\text{-}\log}$-topology for $\delta < 1$, defined in Section \ref{sec2}. Then, we prove that this cocycle is a discontinuity point for the Lyapunov exponents. 

\begin{teorema}\label{teo1}
For $\delta<1$, in the $C_{\delta-\log}$-topology, $A$ can be approximated by cocycles with vanishing Lyapunov exponents. In particular, it is a discontinuity point for the Lyapunov exponents in this topology.    
\end{teorema}

Although the above result is established for a specific cocycle, it exhibits a phenomenon analogous to that described by Bochi–Mañé's theorem. More precisely, there exists a sequence $(B_n)_{n \in \mathbb{N}} \subset C^0(M,\SL(2)) \setminus \mathcal{UH}$ such that $B_n \to A$ in the $C_{\delta-\log}$-topology, and each $B_n$ has vanishing Lyapunov exponents.

We note that the methods employed in the proof of Theorem \ref{teo1} only yield discontinuity of the Lyapunov exponents in the $C_{\delta-\log}$-topology, which is less regular than the Hölder topology. Nevertheless, examples constructed by Bocker and Viana \cite{bockerviana}, Butler \cite{butler}, and examples by Mamani and Saraiva \cite{mamanisaraiva} (see Section \ref{exemploa} for the definition), based on similar arguments, also establish discontinuity in the Hölder topology.

\section{Preliminaries}\label{sec2}
In this section, we introduce the definitions and notation required for the formulation of the main results. The theorem is stated for the full shift on the base, since its definition and properties play a crucial role in the construction of the proof.

Let $M = \{0,1\}^{\mathbb{Z}}$ be the space of bi-infinite sequences over an alphabet of two symbols. For any $\rho \in (0,1)$, we equip $M$ with the metric 
\begin{equation}\label{metric}  
d_{\rho}(x,y) = \rho^{N(x,y)} \quad \text{ where } \quad N(x,y) = \min\{|i| \geq 0 : x_i \neq y_i\}, 
\end{equation}
which generates the product topology on $M$. Since, for each $\rho \in (0,1)$, the resulting metrics are Hölder equivalent and therefore induce the same topology on $M$, we fix $\rho$ from now on.

Let $f \colon M \to M$ be the left shift map defined by $f\big((x_i)_{i \in \mathbb{Z}}\big) = (x_{i+1})_{i \in \mathbb{Z}}$. The map $f$ is a hyperbolic homeomorphism with $1/\rho$ the expansion rate of $f$, where $\rho$  is the constant appearing in the definition of the metric.

Given $p \in (0,1)$, consider the probability measure $\nu =p\delta_1+(1-p)\delta_0$ on $\{0,1\}$ and let $\mu_p = \nu^{\mathbb{Z}}$ be the corresponding Bernoulli product measure on $M$. This construction yields the Bernoulli shift system $(M, f, \mu_p)$.

\subsection{Linear Cocycles}
Consider a continuous map $A \colon M \to \mathrm{SL}(2, \mathbb{R})$ over the compact metric space $(M,d_{\rho})$. The linear cocycle defined by $A$ over the base dynamics $f:M \to M$ is the skew-product $F_A \colon M \times \mathbb{R}^2 \to M \times \mathbb{R}^2$ defined by $$F_A(x, v) = (f(x), A(x)v).$$ Since the base transformation $f$ is fixed, we identify the cocycle $F_A$ with its generator $A$ and refer to $A$ itself as a linear cocycle. 

The cocycle property $A^{n+m}(x)=A^n(f^m(x))A^m(x)$ for all $n,m\in \Z$ and $x \in M$ follows directly from the dynamical definition of the iterated products: 
\begin{equation}\label{lll}
    A^n(x) = 
\begin{cases}
A(f^{n-1}(x)) \cdots A(x) & \text{if } n > 0, \\
\text{Id} & \text{if } n = 0, \\
A(f^{-n}(x))^{-1} \cdots A(f^{-1}(x))^{-1} & \text{if } n < 0.
\end{cases}
\end{equation}

We use the following definition of the norm:
$$\|A^n(x)\|= \sup_{\|v\|_2=1} \|A^n(x)v\|_2,$$
where $\|\cdot\|_2$ denotes the usual norm in $\R^2$.

For every continuous function $A:M \to \SL(2)$, the Furstenberg-Kesten Theorem \cite{furstenberg} guarantees the existence of the following limits $\mu_p$-almost every $x \in M$, which are called the extremal Lyapunov exponents associated with the system $(f,A,\mu_p)$,
\[   \lambda_{+}(A,\mu_p,x) = \lim_{n\rightarrow \infty}\frac{1}{n}\log\|A^n(x)\| \quad \text{ and } \quad \lambda_{-}(A,\mu_p,x) = \lim_{n\rightarrow \infty}\frac{1}{n}\log\|A^n(x)^{-1}\|^{-1}. \]
 Since Lyapunov exponents are $f$-invariant functions and the measure $\mu_p$ is ergodic, they are constant $\mu_p$-almost everywhere. Thus, we simply denote them by $\lambda_+(A,\mu_p)$ and $\lambda_-(A,\mu_p)$. For $\SL(2)$-valued cocycles, the relation between the two Lyapunov exponents holds $\lambda_{+}(A, \mu_p)+\lambda_{-}(A, \mu_p)=0$.
 
 From Lemma 9.1 of \cite{viana}, we observe that $\lambda_+(A,\mu_p)$ and $\lambda_-(A,\mu_p)$ are, respectively, an upper semicontinuous and a lower semicontinuous function of $A$. Moreover, if $\lambda_{+}(A, \mu_p)$ is zero, the same holds for $\lambda_{-}(A, \mu_p)$ by the above relation, and therefore the cocycle $A$ is a continuity point.

\subsection{Topologies} We continue to consider the product space $(M, d_{\rho})$ and a continuous function $A:M \rightarrow \SL(2)$. Since the constant $\rho$ is fixed, we will denote $d$ instead of $d_{\rho}$. We introduce now several classes of regularity for such functions, starting from the usual Hölder continuity and moving gradually toward weaker forms of continuity that appear naturally in the study of linear cocycles with limited smoothness, see \cite{duartekleinsantos} for more details. For that, we only use the fact that $M$ is a metric compact space.

A function $A$ is $\alpha$-\textit{Hölder continuous}, for some $\alpha>0$, if it belongs to the space $C^{\alpha}(M, \SL(2))$, that is, if $\|A\|_{\alpha}< \infty$, where  
$$\|A\|_{\alpha}=\sup_{x\in M}\|A(x)\|+\sup_{x \neq y \in M} \frac{\|A(x)-A(y)\|}{d(x,y)^{\alpha}}.$$
This is the standard regularity condition; it can be interpreted in terms of the Hölder constant of $A$, which quantifies how much $A$ can vary between two points relative to a power of their distance. A smaller constant corresponds to a more regular $A$.

A weaker version of this condition can be defined by replacing the power of the distance with a logarithmic term.

    We say that $A$ is a \textit{weak Hölder continuous} function if $A \in C_{weak}(M, \SL(2))$, that is, if
    $$\|A\|_{weak}=\sup_{x\in M}\|A(x)\|+\sup_{x \neq y \in M} \left\{\|A(x)-A(y)\|\exp\left(\alpha \left(\log \frac{1}{d(x,y)}\right)^{\theta}\right)\right\}< \infty,$$
for some $\alpha, \theta \leq 1$.

We can note that, due to logarithmic term, this definition allows for slower rates of continuity decay as points approach each other.
Also, when $\theta=1$, the estimate above coincides with the classical Hölder continuity, then the weak Hölder condition truly generalizes the standard case.

An even more relaxed form of regularity is obtained by introducing iterated logarithmic terms.

    The function $A$ is said to be $(\gamma, \kappa)$\textit{-log-Hölder continuous} if $A \in C_{(\gamma,\kappa)}(M, \SL(2))$, that is, if 
    $$\|A\|_{(\gamma, \kappa)}=\sup_{x\in M}\|A(x)\|+\sup_{x \neq y \in M} \left\{\|A(x)-A(y)\|\exp\left(\kappa \left(\log\log \frac{1}{d(x,y)}\right)^{\gamma}\right)\right\}<\infty,$$
    where $\gamma, \kappa \geq 1$.

This class consists of functions whose regularity decays at a double-logarithmic rate, allowing for even slower convergence than in the weak Hölder case. When $\gamma=1=\kappa$, the definition reduces to the usual log-Hölder continuity, which we introduce next. Throughout, we will use the notation $1$-log-Hölder continuity to refer to the standard log-Hölder case, since we will work with broader regularity classes that generalize this type of continuity.

    Then, we say that $A$ is a 1-\textit{log-Hölder continuous} function if $A \in C_{1-\log}(M, \SL(2))$, i.e, if 
    \[\|A\|_{1-\log}=\sup_{x\in M}\|A(x)\|+\sup_{x \neq y \in M} \left\{\|A(x)-A(y)\| \left(\log \frac{1}{d(x,y)}\right)\right\}<\infty.\] 

   Finally, for $0<\delta<1$, we define the $\delta$\textit{-log-Hölder continuity} by requiring that $A \in C_{\delta-\log}(M, \SL(2))$, where 
    \[\|A\|_{\delta-\log}=\sup_{x\in M}\|A(x)\|+\sup_{x \neq y \in M} \left\{\|A(x)-A(y)\| \left(\log \frac{1}{d(x,y)}\right)^{\delta}\right\}<\infty.\]
This class provides a continuous transition between log-Hölder and merely continuous functions, reflecting very mild degree of regularity.

These definitions form a natural relation between the spaces describing decreasing levels of regularity.
Indeed, one can verify the following inclusions
 \begin{eqnarray}\label{inclusion}
        C^{\alpha} \subset C_{weak} \subset C_{(\gamma, \kappa)} \subset C_{1-\log} \subset C_{\delta-\log} \subset C^0,
    \end{eqnarray}
showing how each new space properly extends the previous one.
\subsection{Examples in each regularity class}\label{exemploa}
    For any constants $\sigma \geq\eta>1$, consider the cocycle associated with the following continuous function $A_{\sigma \eta}: M \to \SL(2)$ defined by 
    \begin{eqnarray*}
	A_{\sigma\eta}(x)= \left\{
	\begin{array}{ll}
	\begin{pmatrix}
	\eta^{-1} & 0 \\
	0 & \eta 
	\end{pmatrix} &\,\ \mbox{if $x_0=0$} \\
	\\
	\begin{pmatrix}
	\sigma & 0 \\
	0 & \sigma^{-1}
	\end{pmatrix} &\,\ \mbox{if $x_ 0=1$,}
	\end{array}\right.
	\end{eqnarray*}
    whose upper Lyapunov exponent is given by the number $\lambda_{+}(\ase, \mu_{p})=|(1-p)\log \eta - p \log \sigma|$.

    The cocycle $\ase$ is $\alpha$-Hölder continuous for every $\alpha>0$, since it depends only on a finite number of cylinders. Consequently, by the inclusion \eqref{inclusion}, this cocycle belongs to all the spaces defined above.
    
    For every choice of $\sigma$ and $\eta$ and, for all weights $p>0$ such that $p \neq \frac{\log \eta}{\log \sigma \eta}$, Bochi-Mañé's theorem, \cite{bochi} implies that the cocycle $\ase$ is a discontinuity point for the Lyapunov exponents with respect to $\mu_p$ in $C^0$-topology.  
    
    Moreover, without imposing any additional restriction on the parameters $\sigma$ and $\eta$, the cocycle $\ase$ is also a discontinuity point for the Lyapunov exponents in each of these intermediate topologies. The proof of this claim follows from the fact that the norms defining these topologies involve logarithmic terms, while one can be construct a family of cocycles that is exponentially close to the original cocycle in the $\alpha$-norm, as shown in \cite{mamanisaraiva}. Consequently, these cocycles, whose Lyapunov exponents vanish, can be made arbitrarily close to $\ase$ in each of those norms.

    In contrast with those topologies, the situation in the $\alpha$-Hölder topology is more subtle: discontinuity is not the only possible behavior. Indeed, depending on the parameters, there exist regions where the Lyapunov exponents vary continuously, such as the fiber-bunched region (see \cite{bbb}), as well as regions of discontinuity (see \cite{mamanisaraiva}).

    \section{Cocycle with the Identity}\label{sec3}

 
 In this section, we introduce the cocycle associated with a function taking values in diagonal matrices, one of which is the identity matrix. We recall some known results concerning this cocycle and discuss the structural restrictions that such a class of cocycle imposes.

Fix $\sigma>1$. We consider the cocycle associated with the function $\as1:M \rightarrow \SL(2)$ defined by 
\begin{eqnarray}\label{as1def}
	A_{\sigma1}(x)= \left\{
	\begin{array}{ll}
	\begin{pmatrix}
	1 & 0 \\
	0 & 1 
	\end{pmatrix} &\,\ \mbox{if $x_0=0$} \\
	\\
	\begin{pmatrix}
	\sigma & 0 \\
	0 & \sigma^{-1}
	\end{pmatrix} &\,\ \mbox{if $x_ 0=1$.}
	\end{array}\right.
	\end{eqnarray}

Observe that the cocycle $\as1$ is $\alpha$-Hölder continuous for every $\alpha>0$, since it is locally constant, that is, depends only on the first coordinate of the point. Consequently, by the inclusions in~\eqref{inclusion}, $\as1$ belongs to each of the regularity spaces defined in the previous section.

Moreover, the Lyapunov exponents associated with this cocycle are given by \[\lambda_{\pm}(\as1, \mu_{p})=\pm p\log\sigma,\] which are nonzero since $p$ is positive. 

Note that this cocycle is not uniformly hyperbolic, since it remains bounded along some orbit. Indeed, consider the fixed point $\tilde{x} \in M$ consisting entirely of the symbol $0$. Then for every $n \geq 1$, $$\|A^n_{\sigma 1}(\tilde{x})\|=1.$$
Hence, there is no uniform exponential growth of the norm of the iterates of the cocycle. See the hyperbolicity criterion in Chapter 2 of \cite{viana} for more details. 

Nevertheless, for every $\sigma>1$, the cocycle $\as1$ lies on the boundary of the set of uniformly hyperbolic cocycles. Indeed, to see this, consider the family of cocycles $\ase$ defined in Section \ref{exemploa}. When $\eta>1$ is sufficiently close to one, the cocycle $\ase$ is not uniformly hyperbolic, since for $p=\frac{\log \eta}{\log \sigma \eta}$, $\lambda_+(\ase, \mu_p)=0$, while remaining close to $\as1$ in the appropriate topology. On the other hand, when $\eta<1$ is sufficiently close to $1$, the same cocycle $\ase$ is still close to $\as1$, but in this case it is uniformly hyperbolic, with hyperbolic constant equal to $\eta^{-1}\sigma$. Hence, the cocycle $\as1$ sits precisely at the interface between continuity and discontinuity of the Lyapunov exponents, depending on the level of regularity. 

Since the cocycle $\as1$ is not uniformly hyperbolic and has non-zero Lyapunov exponents, by Bochi-Mañé's, $\as1$ is a discontinuity point for the Lyapunov exponents with respect to $\mu_{p}$ in the $C^0$- topology.

In constrast, when we consider the $C^{\alpha}$-topology, Backes, Brown and Butler's Theorem provides a sufficiently condition for continuity: $\as1$ is a continuity point if it satisfies the fiber-bunching condition, that is,  \[\left\|A_{\sigma 1}(x)\right\|\left\|A_{\sigma 1}(x)^{-1}\right\|< \rho^{-
\alpha} \,\ \,\ \forall  x \in M.\]
And this inequality holds precisely if and only if $\sigma^2<\rho^{-\alpha}$. Hence, $\as1$ is a $C^{\alpha}$-continuity point for the Lyapunov exponents with respect to $\mu_{p}$ whenever $\sigma^2<\rho^{-\alpha}$. 

Conversely, when $\sigma^2\geq \rho^{-\alpha}$, it is unknown whether the corresponding cocycle is a continuity or discontinuity point in $C^{\alpha}(M_, \SL(2))$. 

 Rather than restricting ourselves to the Hölder setting, we investigate intermediate topologies between $C^{\alpha}$ and $C^0$, which provide a more flexible framework for analyzing the behavior of the Lyapunov exponents for cocycles near $\as1$. By enlarging the class of admissible cocycles, we can examine how the discontinuity of the Lyapunov exponents behaves when weaker regularity assumptions are considered.
 
In particular, our goal is to understand how the behavior observed in the Hölder topology, such as the examples in Section \ref{exemploa}, where continuity holds only in this topology, changes when one considers weaker topologies. This allows us to understand whether the continuity phenomenon persists in these intermediate settings or whether the discontinuity observed in the $C^0$-topology already appears in topologies weaker than the Hölder one. 
    
    \section{Discontinuity Example in \texorpdfstring{$C_{\delta-\log}$}{}-topology}\label{identityres}
In this section, we prove Theorem~\ref{teo1}, which establishes that the cocycle $\as1$ is a discontinuity point for the Lyapunov exponents in the $C_{\delta-\log}$-topology, for all parameters $\sigma>1$ and $\delta<1$. As previously discussed, this topology is finer than the $C^{0}$-topology.

The proof follows the general strategy: we construct a perturbation of the original cocycle that interchanges the Oseledets subspaces. This property forces the Lyapunov exponents of the perturbed cocycle to vanish, yielding the desired discontinuity. In the present setting, however, the weaker regularity of the $C_{\delta-\log}$-topology allows for greater flexibility, and we are in fact able to construct two distinct perturbations with this property.
    
    \begin{teorema}
\textit{For any $\sigma>1$ and $\delta<1$, the cocycle $\as1$ defined in \eqref{as1def} can be approximated by cocycles with vanishing Lyapunov exponents in the $C_{\delta-\log}$-topology. In particular, $\as1$ is a discontinuity point for the Lyapunov exponents in this topology.}
    \end{teorema}

\begin{proof} First, denote $V_x=\R(1,0)$ and $H_x=\R(0,1)$, respectively, the vertical and horizontal line bundles. Since the cocycle $\as1$ is defined by diagonal matrices, both bundles are invariant under its action. Moreover, these subspaces coincide almost everywhere with the Oseledets subspaces associated with the cocycle.

As we mentioned earlier, our goal is to construct two perturbations of $\as1$, arbitrarily close to it in the $C_{\delta-\log}$-topology, that interchange the Oseledets subspaces. The existence of such perturbations implies that the Lyapunov exponents of the perturbed cocycles vanish.

The general idea is as follows. In the first perturbation, we modify the original cocycle by introducing a sequence of small rotations along the orbit of points contained in a suitably chosen cylinder. These rotations will progressively alter the directions of the invariant  subspaces until the vertical and horizontal ones are exchanged. 

The second perturbation is based on a different geometric mechanism. We begin by applying a small shear that slightly tilts the vector $e_1$ from the horizontal direction. Then, along the orbit, we compose the dynamics with hyperbolic matrices that contract the horizontal direction while expanding the vertical one. After a sufficient number of iterations, we apply a small rotation to the image of $e_1$ aligning it with the vertical axis. Using the same rotation, we then move the vector $e_2$ from the vertical direction to one forming a small angle with the vertical axis. The action of the hyperbolic matrix of the cocycle then amplifies this change, bringing the vector closer to the horizontal direction. Finally, a second shear completes the interchange between the two invariant directions.
\begin{itemize}
    \item \textit{First Perturbation}:\end{itemize} 
    
    Let $k \in \N$ and consider the cylinder $Z_k=[0;0\cdots 01]$ where the symbol $0$ appears k times in $Z_k$. By construction, the colection of sets $\{f^i(Z_k)\}_{i=0}^{k-1}$ are pairwise disjoint, that is, $f^{i}(Z_k)\cap f^j(Z_k)=\emptyset$ for $0\leq i<j \leq k-1$.
   
    We now define a perturbation $B_k: M \rightarrow \SL(2)$ by
\begin{eqnarray}\label{firstp}
B_k(x)= \left\{
\begin{array}{ll}
 \as1(x) R_{\theta_k}\,\ &\mbox{if $x \in \bigcup_{i=1}^{k-1}f^{i}(Z_k)$} \\
\\
\as1(x)\,\ &\mbox{otherwise},
\end{array}\right.
\end{eqnarray}
where $\theta_k=\frac{\pi}{2k}$ and $R_{\theta_k}$ is the rotation matrix, $R_{\theta_k}=\begin{pmatrix}
    \cos(\theta_k) & -\sin(\theta_k)\\
    \sin(\theta_k) & \cos(\theta_k)
\end{pmatrix}$.     

Note that $B_k \in C_{\delta-\log}$, $ \forall k\geq 1$, since $B_k$ is constant on cylinders of diameter at least $\rho^{k}$.

Furthermore, observe that if $x \in Z_k$, then \[B_k^k(x)=R_{k\theta_k}=R_{\frac{\pi}{2}}.\]
In other words, for every $x \in Z_k$, the $k$-th iterate of the cocycle $B_k$ act as a rotation by an angle of $\frac{\pi}{2}$.

Therefore,
\begin{eqnarray*}
  B_k^k(x)H_x=V_{f^k(x)} \,\ \mbox{and} \,\ B_k^k(x)V_x=H_{f^k(x)} \,\ \,\  \forall x\in Z_k.  
\end{eqnarray*}
In particular, the perturbation satisfies the exchange property for every $k\in \Z$. 

We now proceed to show that $B_k$ converges to $\as1$ in this topology. More precisely, we prove that for every $\varepsilon >0$ there is $k_0>0$ such that for all $k \geq k_0$, $$\|B_k - \as1\|_{\delta-\log}< \varepsilon.$$

To this end, fix $\varepsilon>0$ and recall that $$\|B_k-\as1\|_{\delta-\log}=\|B_k-\as1\|_0+\sup_{x \neq y \in M}\left\{\|B_k(x)-\as1(x)-B_k(y)+\as1(y)\|\left(\log\frac{1}{d(x,y)}\right)^{\delta}\right\}.$$ 

We first observe that the first term of the above sum is bounded by $\sigma \frac{\pi}{2k}$, which clearly decays to zero as $k\to \infty$.

To handle the second term, we define, for all $x,y \in M$, $$\hat{T}_k(x,y)=\|B_k(x)-\as1(x)-B_k(y)+\as1(y)\|\left(\log\frac{1}{d(x,y)}\right)^{\delta},$$ and we analyze $\hat{T}_k(x,y)$ by considering the possible relative positions of $x$ and $y$. 

If $x$ and $y$ belong to different cylinders $[0;a]$ with $a\in \{0,1\}$, then their distance satisfies $d(x, y)=1$. Consequently, the logarithmic tem vanishes and we obtain $\hat{T}_k(x,y)=0$.

On the other hand, if $x$ and $y$ are in the same cylinder, $[0;0]$ or $[0;1]$, we have $\as1(x)=\as1(y)$. In this situation, two subcases may occur.

If one of the points belongs to $\bigcup_{i=1}^{k-1}Z_k$ and the other does not, we have the estimates $d(x,y)^{-1}\leq \rho^{-k}$ and $\|B_k(x)-B_k(y)\|=\frac{\pi}{2k}$. Hence,
\begin{eqnarray} \label{thereason}
\hat{T}_k(x,y) \leq \frac{\pi}{2k}k^{\delta}(\log \rho^{-1})^{\delta} \xrightarrow[k \to\infty]{} 0 
\end{eqnarray}
since $\delta<1$. 

In all remaining cases, the matrices $B_k$ and $\as1$ coincide on both points, and thus $\hat{T}_k(x,y)=0$.  

Therefore, for sufficiently large $k$, we conclude that $\|B_k-\as1\|_{\delta-\log}<\varepsilon$. 

It is worth emphasizing that the assumption $\delta<1$ plays a crucial role here: it ensures that term $\frac{\pi}{2k}k^{\delta}(\log \rho^{-1})^{\delta}$ indeed tends to zero, which in turn guarantees convergence in the $\delta$-logarithmic norm. 

Finally, we compute the Lyapunov exponents of $B_k$. For this purpose, we restrict our attention to a subsystem determined by the cylinder $Z_k$. The motivation for this restriction will become clear below. 

Define $\tau_{Z_k}$ as the first return time to $Z_k$ under $f$. That is, for every $x \in Z_k$,
\[ \tau_{Z_k}(x) = \inf\{m \geq 1 : f^m(x) \in Z_k\}.   \]
Let $\pi \colon M \to \{0,1\}$ be the projection onto the $0$-th coordinate. For every $m \geq 1$ and every $x \in M$, the number of occurrences of the symbol $1$ in the first $m$ coordinates of $x$ is given by
\[  S_m(x) = \sum_{i=0}^{m-1} \pi \circ f^{i}(x).  \]
Consequently, for every $x \in Z_k$, we can define
\begin{equation}\label{stau} S_{\tau_{Z_k}}(x) = \sum_{i=0}^{\tau_{Z_k}(x)-1} \pi \circ f^{i}(x).  
\end{equation}

Thus, recalling Definition \ref{lll}, we obtain an induced cocycle $B^{\tau_{Z_k}}_k$ defined for $x \in Z_k$ by the product
\[ B^{\tau_{Z_k}}_k(x) = B_k^{\tau_{Z_k}(x)}(x) := B_k(f^{\tau_{Z_k}(x)-1}(x)) \cdots B_k(f(x))B_k(x).  \]

And, using the definition of the perturbation \eqref{firstp}, we now calculate the induced cocycle $B_k^{\tau_{Z_k}}$:
\begin{align}\label{inducedidentity}
B_k^{\tau_{Z_k}}(x)
&= B_k(f^{\tau_{Z_k}-1}(x))\cdots B_k(f^{k}(x))B_k^k(x) \notag \\
&= \begin{pmatrix}
0 & -\sigma^{S_{\tau_{Z_k}}(x)+1}\\
\sigma^{-S_{\tau_{Z_k}}(x)-1} & 0
\end{pmatrix}.
\end{align}

Furthermore, for every $m \geq 1$, and $x \in Z_k$, we can count the total number of iterations under $f$ required for $x$ to return to $Z_k$ exactly $m$ times and we denote this number by $\tau_{Z_k}^{(m)}$, that is, the $m$-th return time to $Z_k$ under $f$. Specifically, for every $x \in Z_k$,
\begin{equation}\label{mreturntimes}\tau_{Z_k}^{(m)}(x) = \sum_{i=0}^{m-1} \tau_{Z_k}\big(f^{\tau_{Z_k}^{(i)}}(x)\big). \end{equation}
Here, we use the convention that $\tau_{Z_k}^{(0)}(x) = 0$. 

Let $\mu_{Z_k} = \frac{\mu|_{Z_k}}{\mu(Z_k)}$ be the normalized restriction of $\mu$ to $Z_k$, which is invariant under the first return map $f^{\tau_{Z_k}}$. The measure subsystem $(Z_k, f^{\tau_{Z_k}}, \mu_{Z_k})$ is well known to be ergodic, \cite{friedman}. In this context, its relevance lies in the following relation between the Lyapunov exponents:
\begin{equation}\label{relation}
\lambda_{\pm}(B^{\tau_{Z_k}}_k, \mu_{Z_k}) = \frac{\lambda_{\pm}(B_k, \mu)}{\mu(Z_k)}.
\end{equation}
Thus, the problem reduces to showing the vanishing of Lyapunov exponents for the induced system, i.e., proving that $\lambda_{\pm}(B^{\tau_{Z_k}}_k, \mu_{Z_k}) = 0$. And we proceed now to prove it.

Repeating the same procedure for the second return time, we obtain:
\begin{eqnarray}\label{secondinduced}
    B_k^{\tau_{Z_k}^{(2)}}(x)=\begin{pmatrix}
        -\sigma^{S_{\tau_{Z_k}}(f^{\tau_{Z_k}}(x))-S_{\tau_{Z_k}}(x)}&0\\
        0 & -\sigma^{-S_{\tau_{Z_k}}(f^{\tau_{Z_k}}(x))+S_{\tau_{Z_k}}(x)}
    \end{pmatrix}.
\end{eqnarray}
Moreover, for iterates corresponding to even return times, an induction argument on $j\geq 1$ yields
\begin{eqnarray}\label{eveninduced}
    B_k^{\tau_{Z_k}^{(2j)}}(x)=\begin{pmatrix}
        (-1)^{j}\sigma^{c_j(x)}&0\\
        0 & (-1)^j\sigma^{-c_j(x)}
    \end{pmatrix},
\end{eqnarray}
where the function $c_j$ is defined as 
\begin{eqnarray}\label{cjdef}
    c_j(x) = \sum_{i=1}^{j} \left(S_{\tau_{Z_k}}(f^{\tau^{(2i-1)}_{Z_k}}(x)) - S_{\tau_{Z_k}}(f^{\tau^{(2i-2)}_{Z_k}}(x))\right).
    \end{eqnarray}
 
 Then, denoting $\tau_{Z_k}^{(2j)}=:m_j$, note that at the even return times the induced cocycle $B_k^{m_j}(x)$ takes a diagonal form reflecting expansion in one direction and contraction in the transverse direction, or vice versa, depending on the point $x$. 

The Furstenberg-Kesten Theorem, \cite{furstenberg} implies that for $\mu_{Z_k}$-almost every $x \in Z_k$, the following limit exists and it is non-negative since $B^{\tau_{Z_k}}_k$ is an $\mathrm{SL}(2,\mathbb{R})$-valued cocycle,
\begin{equation} \label{limbk} \lambda_+(B^{\tau_{Z_k}}_k, \mu_{Z_k}) = \lim_{m \to \infty} \frac{1}{m}\log \left\|B^{\tau^{(m)}_{Z_k}}_k(x)\right\| \geq 0. \end{equation}

Also, by \eqref{eveninduced}, we obtain the estimate
\begin{eqnarray*}
    \lim_{j \to \infty} \frac{1}{m_j}\log \|B_k^{m_j}(x)\|
    \leq
    \left(\lim_{j \to \infty} \frac{|c_j(x)|}{m_j}\right)\log \sigma.
\end{eqnarray*}

To complete the proof, it remains to establish the key estimate
\begin{eqnarray}\label{ineqkey}
\lim_{j\to\infty} \frac{1}{m_j}\log \left\|B^{m_j}_k(x)\right\| \leq 0.
\end{eqnarray}

For this purpose, we show that
\begin{equation}\label{cjj}
\left(\lim_{j \to \infty} \frac{|c_j(x)|}{m_j}\right)\log \sigma = 0.
\end{equation}

In order to establish inequality \eqref{ineqkey}, we require the following lemma, whose proof is analogous to that of Lemma 3.1 of~\cite{mamanisaraiva}.

\begin{lema}\label{into}
    The function $S_{\tau_{Z_k}}:Z_k\to \R$ defined in Equation \eqref{stau} belongs to $L^1(\mu_{Z_k})$. 
\end{lema}

The next result is a consequence of the previous lemma and the definition of the function $c_j$ given in \eqref{cjdef}.
\begin{lema}
Let $m_j (x)= \tau_{Z_k}^{(2j)}(x)$ be the $2j$-th return time of $x$ to $Z_k$. Then, for $\mu_{Z_k}$-almost every $x\in Z_k$
\[  \limsup_{j \to \infty} \frac{1}{m_j}\log \|B^{m_j}_k(x)\| \leq 0. \]
\end{lema}
\begin{proof}
We are ready to show that Equation \eqref{cjj} holds, from which the result will follow. Let $g = f^{\tau_{Z_k}^{(2)}}$ denote the second return time map. From Equation \eqref{cjdef}, we obtain
\begin{equation}\label{abirkhoff}
    \lim_{j \to \infty} \frac{1}{m_j}|c_j(x)| = \lim_{j \to \infty} \frac{j}{m_j} \cdot \lim_{j \to \infty} \frac{1}{j} \left| \sum_{i=0}^{j-1} S_{\tau_{Z_k}}(g^i(f^{\tau_{Z_k}}(x))) - \sum_{i=0}^{j-1} S_{\tau_{Z_k}}(g^i(x)) \right|.
\end{equation}
By Birkhoff's Ergodic Theorem and Kac's Theorem, we have
\[  \lim_{j \to \infty} \frac{m_j}{j} = \lim_{j \to \infty} \frac{1}{j} \sum_{i=0}^{2j} \tau_{Z_k}(f^{\tau_{Z_k}^{(i)}}(x)) = 2\int_{Z_k} \tau_{Z_k} \, d\mu_{Z_k} = \frac{2}{\mu(Z_k)}.   \]
Since $f^{\tau_{Z_k}}$ is ergodic and Bernoulli, the subsystem $(g, \mu_{Z_k})$ is also ergodic, see \cite{friedman}. By Lemma \ref{into}, $S_{\tau_{Z_k}} \in L^1(\mu_{Z_k})$, and applying Birkhoff's Ergodic Theorem to \eqref{abirkhoff} yields
\[  \lim_{j \to \infty} \frac{1}{j} \sum_{i=0}^{j-1} S_{\tau_{Z_k}}(g^i(f^{\tau_{Z_k}}(x))) = \int_{Z_k} S_{\tau_{Z_k}} \, d\mu_{Z_k} = \lim_{j \to \infty} \frac{1}{j} \sum_{i=0}^{j-1} S_{\tau_{Z_k}}(g^i(x)).  \]
Combining these results, we conclude that
\begin{eqnarray}\label{lima}
\lim_{j \to \infty} \frac{1}{m_j}|c_j(x)| = \frac{\mu(Z_k)}{2} \cdot \left| \int_{Z_k} S_{\tau_{Z_k}} \, d\mu_{Z_k} - \int_{Z_k} S_{\tau_{Z_k}} \, d\mu_{Z_k} \right| = 0. 
\end{eqnarray}
\end{proof}

Since the limit in \eqref{limbk} exists, and $\{m_j\}_{j \in \mathbb{N}} = \{\tau^{(2j)}_{Z_k}\}_{j \in \mathbb{N}}$ is a subsequence of return times $\{\tau^{(m)}_{Z_k}\}_{m \in \mathbb{N}}$, the limit along this subsequence must satisfy $\lim_{j \to \infty} \frac{1}{m_j}\log \|B^{m_j}_k(x)\| = 0$. Consequently, the full limit must agree, that is,  $$\lambda_+(B^{\tau_{Z_k}}_k, \mu_{Z_k}) = 0.$$ This establishes the vanishing of the Lyapunov exponents.

\begin{itemize}
\item \textit{Second Perturbation}:
\end{itemize}

Now, let $k \in \N$ and consider a different cylinder, $W_k=[0;0\cdots01\cdots 1]$ where the number 0 appears $k+1$ times and the number 1 appears $k$ times. In this construction, we intentionally choose the symbol $1$ to appear more frequently than in the previous cylinder $Z_k$, since the hyperbolic matrix will be used to produce the interchange of subspaces.

As before, the first $2k$ iterates of the cylinder are pairwise disjoint, that is, $f^{i}(W_k) \cap f^j(W_k)=\emptyset$ for $0 \leq i<j\leq 2k$. 

Fix $\beta>0$ such that $1>\beta>\delta$. And, finally, we define the cocycle associated to the function $L_k:M \rightarrow \SL(2)$ by 
\begin{eqnarray}\label{second}
L_k(x)= \left\{
\begin{array}{ll}
\as1(x)\begin{pmatrix}
1 & 0 \\
k^{-\beta} & 1
\end{pmatrix} &\,\ \mbox{if $x \in W_k$} \\
\\
\as1(x)\begin{pmatrix}
(1+k^{-\beta})^{-1} & 0 \\
0 & (1+k^{-\beta}) 
\end{pmatrix} &\,\ \mbox{if $x \in \bigcup_{i=1}^{k}f^i(W_k)$}\\
\\
\as1(x)R_{\theta_k}  &\,\ \mbox{if $x \in f^{k+1}(W_k)$} \\
\\
\as1(x)\begin{pmatrix}
1 & 0 \\
\tilde{\gamma}(k) & 1 
\end{pmatrix} &\,\ \mbox{if $x \in f^{2k}(W_k)$}\\
\\
\as1(x)    &\,\ \mbox{otherwise,}
\end{array}\right.
\end{eqnarray}
where $\theta_k$ is such that $$\tan(\theta_k)=\frac{1}{k^{-\beta}(1+k^{-\beta})^{2k}},$$ and 
\begin{equation}\label{gammak}\tilde{\gamma}(k)=\frac{k^{-\beta}(1+k^{-\beta})^{2k}}{\sigma^{2k}}=\frac{1}{\tan({\theta_k})\sigma^{2k}}.
\end{equation}

Observe that, due to the disjointness of the iterates of $W_k$, the function $L_k$ is well defined for all $k \geq 1$. In addition, $L_k$ is $\delta-\log$ Hölder continuous for every $k\geq 1$, since it remains constant on cylinders of diameter at least $\rho^{2k-1}$.

For any $x \in W_k$, by the definition of $L_k$, we have
\begin{eqnarray*}
    L_{k}^{2k+1}(x) &=& L_k(f^{2k}(x))\cdots L_k(x)\\
    &=& \begin{pmatrix}
0 & -\sin(\theta_k)\sigma^{k+1}(1+k^{-\beta})^k \\
\sigma^{-k-1}\frac{(1+k^{-\beta})^{-k}}{\sin(\theta_k)} & 0 
\end{pmatrix}.
\end{eqnarray*}

Hence, for every $x \in W_k$, 
\begin{eqnarray*}
  L_k^{2k+1}(x)H_x=V_{f^{2k+1}(x)} \,\ \mbox{and} \,\ L_k^{2k+1}(x)V_x=H_{f^{2k+1}(x)}. 
\end{eqnarray*}

Once again, the constructed cocycle exhibits the exchange property for every $k \in \Z$. 

We show now that, for $k$ sufficiently large, $L_k$ is close to $\as1$ in the $\delta-\log$ norm. As in the case of the first perturbation, our goal is to prove that for each $\varepsilon>0$ there is $k_0>0$ such that $\|L_k-\as1\|_{\delta-\log}<\varepsilon$ for all $k \geq k_0$.

For this purpose, we aim to estimate
\begin{equation}\label{normaLk}
\begin{array}{rcl}
    \|L_k-\as1\|_{\delta-\log}&=&\|L_k-\as1\|_0\\
    &+&\sup_{x \neq y \in M}\left\{\|(L_k-\as1)(x)-(L_k+\as1)(y)\|\left(\log\frac{1}{d(x,y)}\right)^{\delta}\right\}.
    \end{array}
    \end{equation}

    We begin by analyzing the first term of the sum, $\|L_k-\as1\|_0$. 
    Recall that $\sigma>1$, $\tan(\theta_k)=\frac{k^{\beta}}{(1+k^{-\beta})^{2k}}$ and $\tilde{\gamma}(k)=\frac{(1+k^{-\beta})^{2k}}{k^{\beta}\sigma^{2k}}$. From these definitions, we have
\begin{eqnarray}\label{gamma}
    \lim_{k\to \infty}\tilde{\gamma}(k)=\lim_{k\to \infty}\frac{1}{k^{\beta}}\left(\frac{1+k^{-\beta}}{\sigma}\right)^{2k}=0,
    \end{eqnarray}
    \begin{eqnarray}\label{theta}
    \lim_{k\to\infty} \tan(\theta_k)=\lim_{k\to \infty}\frac{k^{\beta}}{(1+k^{-\beta})^{2k}}=0.
\end{eqnarray}

Since $\|L_k-\as1\|_0 \leq \sigma \mbox{max}\{k^{-\beta}\mbox{,} \,\ \tilde{\gamma}(k)\mbox{,} \,\ \theta_k\}$, we conclude from (\ref{gamma}) and (\ref{theta}) that the uniform distance between $L_k$ and $\as1$ tends to zero. In particular, for every $\varepsilon>0$ there exists $\tilde{k}>0$ such that $$\|L_k-\as1\|_0<\varepsilon  \,\ \mbox{ for all  }k\geq \tilde{k}.$$

In order to conclude the proof, we now estimate the second term in the $\delta-\log$ norm of \eqref{normaLk}. Once again, we analyze all possible configurations of the points $x$ and $y$ in $M$. 

Let us denote $$\tilde{T}_k(x,y)=\|L_k(x)-\as1(x)-L_k(y)+\as1(y)\|\left(\log\frac{1}{d(x,y)}\right)^{\delta}.$$

If $x$ and $y$ belong to different cylinders, then $d(x,y)=1$ and, consequently, the logarithmic factor vanishes. In this case, we immediately have $\tilde{T}_k(x,y)=0$.

If, on the other hand, $x$ and $y$ belong to the same cylinder, then $\as1(x)=\as1(y)$, hence only the variation of $L_k$ contributes to the estimate of $\tilde{T}_k(x,y)$. We now analyze this variation by considering, separately, both cases about $x$ and $y$:
\begin{itemize}
\item Case 1: $x, y \in [0;0]$

In this case, since both points share the same initial symbol $0$, we will examine the possible positions of $x$ and $y$ relative to the sets $W_k$ and its iterates $f^i(W_k)$ with $0<i\leq k$.
    \begin{enumerate}
        \item $x\in W_k$ and $y \notin \bigcup_{i=0}^{k}f^{i}(W_k)$.
  
  Notice that their distance satisfies $d(x,y)^{-1}\leq \rho^{-2k-1}$ and by the construction of $L_k$, the difference between its values at these points is small, namely $\|L_k(x)-L_k(y)\|\leq k^{-\beta}$. Combining these two estimates yields
    $$\tilde{T}_k(x,y)\leq \frac{(2k+1)^{\delta}(\log\rho^{-1})^{\delta}}{k^{\beta}}=(\log\rho^{-1})^{\delta}\left(\frac{2k+1}{k^{\frac{\beta}{\delta}}}\right)^{\delta}.$$
    Since $\frac{\beta}{\delta}>1$, the exponent of $k$ in the denominator dominates, and it follows that $$\lim_{k \to \infty}\tilde{T}_k(x,y)=0.$$
    \item $x \in \bigcup_{i=1}^{k}f^{i}(W_k)$ and $y \notin \bigcup_{i=0}^{k}f^{i}(W_k)$.
    
In this configuration, $x$ belong to one of the first $k$ forward images of $W_k$, while $y$ remains outside this union. The same reasoning applies, we have $d(x,y)^{-1}\leq \rho^{-2k}$ and $\|L_k(x)-L_k(y)\|\leq k^{-\beta}$. Hence, 
$$\tilde{T}_k(x,y)\leq \frac{k^{\delta}(2\log\rho^{-1{}})^{\delta}}{k^{\beta}}.$$
Since $\beta>\delta$, the right-hand side tends to zero, then
$$\lim_{k\to \infty}\tilde{T}_k(x,y)=0.$$

    \item $x\in W_k$ and $y \in \bigcup_{i=1}^{k}f^{i}(W_k)$.
    
Here, both points belong to the orbit of $W_k$, but under different iterates. From the size of the cylinder $W_k$, we have $d(x,y)^{-1}\leq \rho^{-k}$ and by the definition of the perturbation, it follows that $\|L_k(x)-L_k(y)\|\leq \frac{2}{k^{\beta}}$. Therefore,
$$\tilde{T}_k(x,y)\leq 2(\log\rho^{-1})^{\delta}\frac{k^{\delta}}{k^{\beta}}.$$

Using again $\beta>\delta$, the limit of this expression as $k\to \infty$ is zero: 
    $$\lim_{k\to \infty}\tilde{T}_k(x,y)\leq \lim_{k\to\infty}2(\log\rho^{-1})^{\delta}\frac{k^{\delta}}{k^{\beta}}=0.$$
    \item All other configurations of $x, y\in [0;0]$
    
    For any other configuration within this cylinder, the function $L_k$ takes the same value at $x$ and $y$, hence $\tilde{T}_k(x,y)=0$.
    \end{enumerate}
\end{itemize}

\begin{itemize} 
\item Case 2: $x, y \in [0;1]$. 

We now consider the situation in which both points $x$ and $y$ begin with the symbol $1$. As in the previous case, we examine how the function $L_k$ varies on different subsets of the space and estimate $\tilde{T}_k(x,y)$ accordingly.

If $x$ and $y$ belong to the same image $f^{i}(W_k)$ for some $i\in \{k+1,\cdots,2k\}$, then $L_k(x)=L_k(y)$ and consequently $\tilde{T}_k(x,y)=0$. The same conclusion holds if neither $x$ nor $y$ belong to $f^{k+1}(W_k)\cup f^{2k}(W_k)$, since $L_k$ takes constant values on the complement of these sets.

Hence, we only need to analyze the cases in which one (or both) of the points lies in $f^{k+1}(W_k)$ or $f^{2k}(W_k)$. Then, the remaining possibilities are the following:
\begin{enumerate}
    \item $x \in f^{k+1}(W_k)$ and $y \notin f^{k+1}(W_k) \cup f^{2k}(W_k)$.
    
      Here, we have $d(x,y)^{-1}\leq \rho^{-k-1}$ and the variation of $L_k$ in this region satisfies $\|L_k(x)-L_k(y)\|\leq \sigma \frac{k^{\beta}}{(1+k^{-\beta})^{2k}}$. Combining these two estimates, we obtain 
      $$\tilde{T}_k(x,y)\leq \frac{\sigma k^{\beta}(k+1)^{\delta}(\log\rho^{-1})^{\delta}}{(1+k^{-\beta})^{2k}}.$$
      Since the denominator grows exponentially while the numerator grows only polynomially, this expression tends to zero as $k\to \infty$. Thus
    $$\lim_{k\to \infty}\tilde{T}_k(x,y)=0.$$
    \item $x \in f^{2k}(W_k)$ and $y \notin f^{k+1}(W_k) \cup f^{2k}(W_k)$. 
    
    By construction, the symbolic distance between the points satisfies $d(x,y)^{-1}\leq \rho^{-2k-1}$ and the difference of $L_k$ along these two points is controlled by $\|L_k(x)-L_k(y)\|\leq \sigma \tilde{\gamma}(k)$. Using this inequality and the definition of $\tilde{\gamma}(k)$ in \eqref{gammak}, we obtain
    $$ \tilde{T}_k(x,y)\leq \sigma (\log\rho^{-1})^{\delta}\left(\frac{2k+1}{k^{\frac{\beta}{\delta}}}\right)^{\delta}  \left(\frac{1+k^{-\beta}}{\sigma}\right)^{2k}.$$ 
    
    The first factor decays polynomially because $\beta>\delta$, and the second decays exponentially since $\sigma>1$. Hence, both terms ensure that 
$$\lim_{k\to \infty}\tilde{T}_k(x,y)=0.$$
    \item $x \in f^{k+1}(W_k)$ and $y \in f^{2k}(W_k)$.
    
    In this configuration, both points belong to images of $W_k$ where the perturbation acts, but at different times. 
    
    Since their symbolic sequences differs in the first coordinate, we have $N(x,y)=1$ and thus, $d(x,y)^{-1}=\rho^{-1}$. Moreover, the difference of $L_k$ at these points satisfies $\|L_k(x)-L_k(y)\|\leq \sigma\left(\tilde{\gamma}+\frac{k^{\beta}}{(1+k^{-\beta})^{2k}}\right)$. Therefore,
    $$\tilde{T}_k(x,y)\leq  \sigma(\log\rho^{-1})^{\delta}\left(\tilde{\gamma}+\frac{k^{\beta}}{(1+k^{-\beta})^{2k}}\right),$$
    and since both terms inside the parentheses vanish as $k\to \infty$, we conclude that 
$$\lim_{k\to \infty}\tilde{T}_k(x,y)=0.$$
\end{enumerate}
\end{itemize}

In summary, we have established that  \begin{eqnarray*}
\lim_{k\to \infty}\tilde{T}_k(x,y)=0 \mbox{ for every }x, y \in M.
\end{eqnarray*}
Combining this with the previous estimate on the uniform norm, we may choose $k_0>\tilde{k}$ such that 
\begin{eqnarray*}
    \|L_k-\as1\|_{\delta-\log}<\varepsilon \,\  \,\ \forall k\geq k_0.
\end{eqnarray*}

As we did previously, we can compute the first and second return times, in order to determine the Lyapunov exponets of the perturbation using their definition. Thus, for $\mu_{p}$-almost every point $x \in M$, we have
\begin{eqnarray}
    L_{k}^{\tau_{W_k}}(x) = \begin{pmatrix}
0 & \sin(\theta_k)\sigma^{S_{\tau_{W_k}}(x)}(1+k^{-\beta})^k \\
\sigma^{-S_{\tau_{W_k}}(x)}\frac{(1+k^{-\beta})^{-k}}{\sin(\theta_k)} & 0 
\end{pmatrix}.
\end{eqnarray}
Similarly, the second return time is given by
\begin{eqnarray*}
    L_{k}^{\tau_{W_k}^{(2)}}(x) = \begin{pmatrix}
\sigma^{\tilde{c}_1(x)} &0 \\
0 & \sigma^{-\tilde{c}_1(x)}
\end{pmatrix}.
\end{eqnarray*}
More generally, by recursion, for every $j\geq 1$ we obtain
\begin{eqnarray}
    L_{k}^{\tau_{W_k}^{(2j)}}(x) = \begin{pmatrix}
\sigma^{\tilde{c}_j(x)} &0 \\
0 & \sigma^{-\tilde{c}_j(x)}
\end{pmatrix},
\end{eqnarray}
where the function $\tilde{c}_j$ is the same as that defined in \eqref{cjdef}, but now the return times are taken with respect to the cylinder $W_k$, that is, 
\begin{eqnarray*}\label{cjdef2}
    \tilde{c}_j(x) = \sum_{i=1}^{j} \left(S_{\tau_{W_k}}(f^{\tau^{(2i-1)}_{W_k}}(x)) - S_{\tau_{W_k}}(f^{\tau^{(2i-2)}_{W_k}}(x))\right).
    \end{eqnarray*}

Following the same reasoning as in the previous perturbation, and by the very definition of the perturbation, we conclude that the Lyapunov exponents of $L_k$ vanish. In other words, for all $k\geq 1$, 
\[\lambda_+(L_k, \mu_{p})=0 \]

Therefore, we have proved that for every $\varepsilon>0$, one can construct a $\delta-\log$-Hölder continuous cocycle $L_k$ such that $\lambda_{+}(L_k, \mu_{p})=0$ and is arbitrarily close to the locally constant cocycle $\as1$, $$\|L_k-\as1\|_{\delta-\log}<\varepsilon.$$

This completes the construction and the proof.
\end{proof}

\bibliographystyle{plain}
\bibliography{ref}

\end{document}